\documentclass[12pt]{amsart}
\usepackage{amssymb, amsmath, amsfonts, microtype, cite, tikz, graphicx, tabu}
\usepackage[colorlinks]{hyperref}
\usepackage[capitalize]{cleveref}
\usepackage[a4paper]{geometry}
\usetikzlibrary{calc, decorations.pathmorphing, decorations.footprints}

\newtheorem{thm}{Theorem}[section]
\newtheorem{cor}[thm]{Corollary}
\newtheorem{lem}[thm]{Lemma}

\theoremstyle{definition}

\newtheorem{step}{Step}

\newtheorem{dfn}[thm]{Definition}
\newtheorem{eg}[thm]{Example}

\theoremstyle{remark}

\makeatletter \@addtoreset{equation}{section} \makeatother

\def\={\;=\;}
\def\<{\;<\;}
\def\>{\;>\;}

\def\bg#1{\bigl({#1}\bigr)}
\def\Bg#1{\Big({#1}\Bigr)}
\def\bgg#1{\biggl({#1}\biggr)}

\def\C{\mathbb{C}}

\def\N{\mathbb{N}}
\def\Q{\mathbb{Q}}

\def\R{\mathbb{R}}

\def\Z{\mathbb{Z}}

\def\eqrl{\quad\iff\quad}
\def\rmand{\quad\hbox{ and }\quad}

\newcommand{\prl}{\mathrel{\!/\mkern-5mu/\!}}

\DeclareMathOperator\Arg{\mathrm{Arg}}

\crefname{case}{Case}{Cases}
\crefname{case}{Case}{Cases}
\creflabelformat{case}{~\upshape#2#1#3}
\crefname{step}{Step}{Steps}
\crefname{step}{Step}{Steps}
\creflabelformat{step}{~\upshape#2#1#3}
\crefname{caseclm}{Case}{Cases}
\crefname{caseclm}{Case}{Cases}
\creflabelformat{caseclm}{~\upshape#2#1#3}
\crefname{casethm}{Case}{Cases}
\crefname{casethm}{Case}{Cases}
\creflabelformat{casethm}{~\upshape#2#1#3}
\crefname{clm}{Claim}{Claims}
\Crefname{clm}{Claim}{Claims}
\creflabelformat{claim}{~\upshape#2#1#3}
\crefname{cond}{Condition}{Conditions}
\Crefname{cond}{Condition}{Conditions}
\creflabelformat{cond}{~\upshape(#2#1#3)}
\crefname{def}{Definition}{Definitions}
\Crefname{def}{Definition}{Definitions}
\creflabelformat{def}{~\upshape(#2#1#3)}
\crefname{eqrl}{Equivalence}{Equivalences}
\Crefname{eqrl}{Equivalence}{Equivalences}
\creflabelformat{eqrl}{~\upshape(#2#1#3)}
\crefname{eg}{Example}{Examples}
\Crefname{eg}{Example}{Examples}
\creflabelformat{eg}{~\upshape #2#1#3}
\crefname{fml}{Formula}{Formulas}
\crefname{fml}{Formula}{Formulas}
\creflabelformat{fml}{~\upshape(#2#1#3)}
\crefname{ineq}{Ineq.}{Ineqs.}
\Crefname{ineq}{Inequality}{Inequalities}
\creflabelformat{ineq}{~\upshape(#2#1#3)}
\crefname{prob}{Problem}{Problems}
\Crefname{prob}{Problem}{Problems}
\creflabelformat{prob}{~\upshape#2#1#3}
\crefname{rec}{Recurrence}{Recurrences}
\Crefname{rec}{Recurrence}{Recurrences}
\creflabelformat{rec}{~\upshape(#2#1#3)}
\crefname{rl}{Relation}{Relations}
\crefname{rl}{Relation}{Relations}
\creflabelformat{rl}{~\upshape(#2#1#3)}
\crefname{rst}{Result}{Results}
\Crefname{rst}{Result}{Results}
\creflabelformat{rst}{~\upshape(#2#1#3)}
\crefname{lem}{Lemma}{Lemmas}
\Crefname{lem}{Lemma}{Lemmas}
\creflabelformat{lem}{~\upshape#2#1#3}
\crefname{ini}{Initiation}{Initiations}
\Crefname{ini}{Initiation}{Initiations}
\creflabelformat{ini}{~\upshape(#2#1#3)}
\crefname{thm}{Theorem}{Theorems}
\Crefname{thm}{Theorem}{Theorems}
\creflabelformat{thm}{~\upshape#2#1#3}
\crefname{cor}{Corollary}{Corollaries}
\Crefname{cor}{Corollary}{Corollaries}
\creflabelformat{cor}{~\upshape#2#1#3}
\crefname{sec}{\S\!}{\S\!}
\Crefname{sec}{Section}{Sections}
\creflabelformat{sec}{~\upshape#2#1#3}
\crefname{ssec}{\S\!}{\S\!}
\Crefname{ssec}{Subsection}{Subsections}
\creflabelformat{ssec}{~\upshape#2#1#3}
\crefname{sssec}{\S\!}{\S\!}
\Crefname{sssec}{Subsubsection}{Subsubsections}
\creflabelformat{sssec}{~\upshape#2#1#3}

\crefrangeformat{equation}{Eqs.\!~(#3#1#4) ---~(#5#2#6)}
\crefrangeformat{ineq}{Ineqs.\!~(#3#1#4) ---~(#5#2#6)}
\crefrangeformat{eqrl}{Equivalences\!~(#3#1#4) ---~(#5#2#6)}

\allowbreak
\allowdisplaybreaks

\def\W{\mathcal{W}}

\begin{document}

\title[]{{Common zeros of polynomials \\[5pt]
satisfying a recurrence of order two}}

\author{Dannielle D.D. Jin}
\address{
School of Mathematics and Statistics, Beijing Institute of Technology, 102488 Beijing, P. R. China
}
\email{ddj.combin@gmail.com}

\author[D.G.L. Wang]{David G.L. Wang$^\dag$$^\ddag$}
\address{
$^\dag$School of Mathematics and Statistics, Beijing Institute of Technology, 102488 Beijing, P. R. China\\
$^\ddag$Beijing Key Laboratory on MCAACI, Beijing Institute of Technology, 102488 Beijing, P. R. China}
\email{david.combin@gmail.com}


\begin{abstract}
We give a characterization of common zeros of a sequence of univariate polynomials $W_n(z)$ defined by a recurrence of order two with polynomial coefficients, and with $W_0(z)=1$. Real common zeros for such polynomials with real coefficients are studied further. This paper contributes to the study of root distribution of recursive polynomial sequences.
\end{abstract}


\maketitle                   

%\tableofcontents
\parskip 8pt

\section{Introduction}
The root distribution of a single polynomial is a long-standing topic all along the history of mathematics; 
see Rahman and Schmeisser's book \cite{RS02B}. 
As Gian-Carlo Rota~\cite{Rota85} left to us, 
``The one contribution of mine that I hope will be remembered has consisted in just pointing out that all sorts 
of problems of combinatorics can be viewed as problems of location of the zeros of certain polynomials ...''. 
Apart from combinatorics, Robin Pemantle~\cite{Pem13} presented that ``The `geometry' of a polynomial 
refers to the geometry of its zero set. In fact, for an algebraic geometer, a polynomial is equated with its zero set.''
For example, both the real-rootedness and the stability of polynomials attract; 
see Stanley~\cite[\S 4]{Sta00} and Borcea and Br\"anden~\cite{BB09}.

This paper concerns common zeros of distinct polynomials, 
as a special circumstance of root distribution.
In their study of the common zero structure of rational matrix functions,
Lerer and Rodman \cite{LR96} pointed that
``the problem of finding common zeros of polynomials ...  has kept the attention of mathematicians for centuries. 
The first motivation came from the problem of determining the intersection points 
of two algebraic curves \cite{Die85B}. 
Later, the study of the asymptotic stability of linear differential and finite-difference equations 
with constant coefficients has created a great interest in the problem of 
determining the location of the zeros of a polynomial with respect to the imaginary axis and the unit circle 
\cite{KN81}. 
The latter problem can be viewed as a specific common zeros problem ...''.
Common zeros of kinds of special functions has also received much attention,
such as Bessel functions~\cite{BK78}, Legendre's associated functions~\cite{Lac84}, 
and $L$-functions~\cite{Li12}.

Motived by the LCGD conjecture from topological graph theory,
Gross, Mansour, Tucker and the second author~\cite{GMTW16-01,GMTW16-10}
studied the root distribution of polynomials satisfying recurrences 
of order two, subject to some conditions on the polynomial coefficients. 

\begin{dfn}
We call a sequence $\W=\{W_n(z)\}_{n\ge0}$ of polynomials 
a {\em recursive polynomial sequence of order two} if
\begin{equation}\label[rec]{rec:order2}
W_n(z)\=A(z)W_{n-1}(z)+B(z)W_{n-2}(z),\qquad\text{for $n\ge 2$},
\end{equation}
where $A(z)$ and $B(z)$ are polynomials with complex coefficients, independent of $n$. 
We call a complex number $c$ a {\em common zero} of $\W$ if 
\begin{equation}\label{def:cr}
W_{s}(c)=W_{t}(c)=0\qquad\text{for some $s\ne t$}.
\end{equation}
We call $\W$ {\em normalized} if $W_0(z)=1$.
\end{dfn}

Orthogonal polynomials are involved in our study,
since they can be defined by \cref{rec:order2} with a linear polynomial coefficient $A(z)$
and a constant polynomial coefficient $B(z)$; 
see Andrews, Richard, and Ranjan's book~\cite{ARR99B}
for basic information on orthogonal polynomials.
Common zeros of $\binom{n-d+1}{d}$ multivariate orthogonal polynomials 
of degree $n$ in $d$ variables has been used~\cite{Xu94} to 
study joint eigenvalues of truncated block Jacobi matrices.

Determining the set of limit points is one of the critical topics in the study of root distribution;
see \cite{BKW75,BG07}.
Any limit point of the union of zero sets of all polynomials is associated with
a sequence of convergent distinct zeros, except when the limit point itself is a common zero. 
As an application of \cref{cor:period}, i.e., of that any common zero of polynomials defined by \cref{rec:order2}
occurs periodically, we conclude that every common zero is a limit point of the union of zero sets.

We organize this paper as follows.
In \cref{sec:complex} we deduce a characterization 
for common zeros of polynomials satisfying \cref{rec:order2},
with neither restriction on the degrees of the polynomial coefficients $A(z)$ and $B(z)$,
nor limitation on the realities of coefficients of $A(z)$ and $B(z)$; see \cref{thm:cr}.
\Cref{sec:real} consists of some further results for real common zeros 
when $A(z)$ and $B(z)$ are real polynomials.

\section{The characterization for common zeros}\label[sec]{sec:complex}

Throughout this paper, we use the notation $\W=\{W_n(z)\}_{n\ge0}$
to denote a normalized recursive polynomial sequence of order two.
First of all, we consider the roots of the polynomial coefficients $B(z)$.

\begin{lem}\label{lem:AB=0}
If a common zero $c$ of $\W$ satisfies $B(c)=0$, 
then $W_n(c)=0$ for all $n\ge2$.
\end{lem}

\begin{proof}
Suppose that $B(c)=0$.
Then \cref{rec:order2} implies $W_n(c)=A^{n-1}(c)W_1(c)$ for all $n\ge2$. Since~$c$ is common, we find $A(c)W_1(c)=0$. 
Thus $W_n(c)=0$ for all $n\ge 2$.
This completes the proof.
\end{proof}

In the remaining of this paper, we let $c\in\C$ such that $B(c)\ne0$.
Note that the sequence $\{W_n(c)\}_{n\ge0}$ of complex numbers satisfies
the recurrence 
\[
W_n(c)\=A(c)W_{n-1}(c)+B(c)W_{n-2}(c).
\]
This leads us to find the general solution 
for a sequence of numbers defined by a recurrence relation of order two.
%For convenience, we use the same symbol $W_n$ to denote a complex number. There will be no risk of confusion with the symbol $W_n(z)$ that indicates a polynomial.

For any $z\in\C\backslash\{0\}$,
we denote the principle value of the argument of $z$ 
by $\Arg(z)$, and restrict $\Arg(z)\in(-\pi,\,\pi]$ as usual.
We use the notation 
\[
\sqrt{z}\=\sqrt{|z|}\,e^{i\cdotp\Arg(z)/2},
\]
so that $\sqrt{z}$ is a complex number lying
either on the right-half open plane or on the positive imaginary axis.

\begin{lem}\label{lem:00}
Let $A,B\in\C$.
Let $\{W_n\}_{n\ge0}$ be a sequence of numbers defined by $W_0=1$
and $W_n=AW_{n-1}+BW_{n-2}$ for $n\ge 2$.
Then for any $n\ge 1$, we have 
\[
W_n\=\begin{cases}
\displaystyle \frac{A^{n-1}}{2^n}\bg{A+n(2W_1-A)},
&\textrm{if $\Delta=0$},\\[8pt]
\displaystyle {g^+(A+\sqrt{\Delta}\,)^n-g^-(A-\sqrt{\Delta}\,)^n
\over 2^n\cdot \sqrt{\Delta}},&\textrm{if $\Delta\neq0$},
\end{cases}
\]
where $\Delta=A^2+4B$ and $g^\pm=(2W_1-A\pm\sqrt{\Delta})/2$.
\end{lem}

\begin{proof}
Given the initial complex numbers $W_0$ and $W_1$,
the recurrence relation uniquely determines every complex number $W_n$. Thus it suffices to verify that the solution satisfies the recurrence, which is routine. 
\end{proof}

A real number version of \cref{lem:00} can be 
found in \cite{Bat67B,GMTW16-01,GMTW16-10}.
Inspired from \cref{lem:00}, we define
\[
\Delta(z)=A^2(z)+4B(z)
\rmand
g^\pm(z)=\frac{2W_1(z)-A(z)\pm\sqrt{\Delta(z)}}{2}.
\]

\begin{lem}\label{lem:<>0}
Let $c\in\C$ such that $B(c)\ne0$. Then we have the following.
\begin{itemize}
\itemsep 5pt
\item[(i)]
$A(c)\pm\sqrt{\Delta(c)}\ne0$.
\item[(ii)]
If $W_n(c)=0$ for some $n$, and $\Delta(c)\ne0$, then $g^\pm(c)\ne0$.
\item[(iii)]
If $c$ is a common zero of $\W$, 
then $\Delta(c)\ne0$.
\end{itemize} 
\end{lem}

\begin{proof}
We show them individually.

\noindent{\bf (i).} 
We have 
$
0\ne-4B(c)
=A^2(c)-\Delta(c)
=\bg{A(c)+\sqrt{\Delta(c)}\,}\bg{A(c)-\sqrt{\Delta(c)}\,}$.

\noindent{\bf (ii).}
Suppose that $\Delta(c)\ne0$.
Since $W_n(c)=0$, \cref{lem:00} implies
\[
g^+(c)\bg{A(c)+\sqrt{\Delta(c)}\,}^n
=g^-(c)\bg{A(c)-\sqrt{\Delta(c)}\,}^n.
\]
If $g^+(c)g^-(c)=0$, 
then the above equation implies that $g^+(c)=g^-(c)=0$.
Consequently, the definition of $g^\pm(z)$ implies that $\Delta(c)=0$, a contradiction.
%This proves $g^\pm(c)\ne0$.

\noindent{\bf (iii).}
By \cref{lem:00} and the definition of common zeros, we infer that
\[
A(c)+s\bg{2W_1(c)-A(c)}
\=A(c)+t\bg{2W_1(c)-A(c)}
\=0
\]
for some distinct integers $s$ and $t$.
Thus $A(c)=0$ and
$\Delta(c)=A^2(c)+4B(c)\ne0$.
\end{proof}

By \cref{lem:<>0}, for any $c\in\C$ such that $B(c)\Delta(c)\ne0$, 
we can define 
\begin{align*}
u&\=\frac{g^-(c)}{g^+(c)}
\=\frac{2W_1(c)-A(c)-\sqrt{\Delta(c)}}{2W_1(c)-A(c)+\sqrt{\Delta(c)}},\rmand\\[5pt]
v&\=\frac{A(c)+\sqrt{\Delta(c)}}{A(c)-\sqrt{\Delta(c)}}.
\end{align*}

Next is a characterization of common zeros, 
read from \cref{lem:00} directly.
\begin{thm}\label{thm:pr}
Let $c\in\C$ such that $B(c)\ne0$.
Then $c$ is a common zero of $\W$ 
if and only if $\Delta(c)\ne0$ and there exist integers 
\[
p=\min\{\ell\in\Z^+\colon v^\ell=1\}\rmand
r=\min\{\ell\in\Z^+\colon v^\ell=u\}.
\]
\end{thm}

\begin{proof}
Since $B(c)\ne0$, the number $v$ is well-defined from \cref{lem:<>0}.
We show the necessity and sufficiency individually. 

\noindent{\bf Necessity.}
Suppose that $c$ is a common zero of $\W$.
By \cref{lem:<>0}, we have $\Delta(c)\ne0$ and $g^\pm(c)\ne0$.
Thus the number $u$ is well-defined.
From \cref{lem:00}, we see that 
\begin{equation}\label[eqrl]{eqrl:vn=u}
W_n(c)=0
\eqrl
u=v^n.
\end{equation}
By \cref{eqrl:vn=u} and the definition of common zeros, 
we have $u=v^s=v^t$ for some distinct integers $s$ and $t$,
which implies the existence of the desired integer~$r$ immediately.
Since $uv\ne0$, we have $v^{|s-t|}=1$, 
and obtain the existence of the desired integer~$p$.

\noindent{\bf Sufficiency.}
Suppose that $\Delta(c)\ne0$.
By the existence of the number $r$, 
we see that the number $u=v^r$. 
From definition of $p$, we have $v^p=1$.
Thus for any $k\in\N$, one infers that
\[
v^{kp+r}\=(v^p)^k\cdot v^r\=1^k\cdotp u\=u.
\]
By \cref{eqrl:vn=u}, we have $W_{kp+r}(c)=0$. 
Thus $c$ is a common zero of $\W$.
\end{proof}

The numbers $p$ and $r$ depend on the number $c$, 
and on the polynomials $A(z)$, $B(z)$, and $W_1(z)$. 
They will be used frequently.

\begin{thm}\label{thm:n}
If $c$ is a common zero of $\W$ such that $B(c)\ne0$,
then $p\ge2$, $r\in\{1,2,\ldots,p-1\}$, and
\[
W_n(c)=0
\eqrl
n\equiv r\pmod{p}.
\]
\end{thm}

\begin{proof}
By \cref{lem:<>0}, we have $\Delta(c)\ne0$.
Thus the numbers $u$ and $v$ are well-defined and nonzero.
By \cref{thm:pr}, the integers $p$ and $r$ are well-defined.

If $p=1$, then $v=1$ since $v^p=1$.
Thus $\Delta(c)=0$ from the definition of $v$, a contradiction.
This proves $p\ge2$. 
Note that $v^{r-p}=v^r=u$.
By the minimality of~$r$, we have $r\le p$. 
If $r=p$, then $u=1$, 
which implies $\Delta(c)=0$ from the definition of~$u$, a contradiction.
This proves $r\in\{1,2,\ldots,p-1\}$.

If $n\equiv r\pmod{p}$, 
then $v^n=v^r=u$ from the definitions of $p$ and $r$.
By \cref{eqrl:vn=u}, we obtain $W_n(c)=0$.
Conversely, suppose that $u=v^n$. 
Let $r''\in\{0,1,\ldots,p-1\}$ be the unique integer 
such that $n\equiv r''\pmod{p}$.
Then $v^r=v^n=v^{r''}$.
It follows that \hbox{$v^{|r-r''|}=1$}. 
Since $|r-r''|\le p-1$, the minimality of~$p$ implies $r=r''$.
This proves the desired equivalence.
\end{proof}

We recognize the following periodicity result immediately.

\begin{cor}\label{cor:period}
The subscripts $n$ of polynomials $W_n(z)$ 
which share a common zero form an arithmetic progression.
\end{cor}
\begin{proof}
Immediate from \cref{lem:AB=0,thm:n}.
\end{proof}

Denote $\Q\pi=\{x\pi\colon x\in\Q\}$, 
where $\Q$ stands for the field of rational numbers.
For any $z\in\C\backslash\{0\}$, we denote $p^*(z)=p$, 
if $\Arg(z)=2q\pi/p$ for some integers $p\in\Z^+$ and $q\in\Z$ 
such that $(p,q)=1$.
We use the binary relation symbol $\prl$ to denote the parallel relation, 
$\perp$ the  perpendicular relation, 
and $|$ the divisible relation.

\begin{thm}\label{thm:cr}
Let $c\in\C$ such that $B(c)\ne0$.
Then $c$ is a common zero of $\W$ if and only if 
the following conditions hold:
\begin{itemize}
\itemsep 5pt
\item[(i)]
$A^2(c)=x\cdotp B(c)$ for some $-4<x\le 0$;
\item[(ii)]
$A(c)\prl W_1(c)$;
\item[(iii)]
$\Arg(u),\Arg(v)\in\Q\pi$ and $p^*(u)\,|\,p^*(v)$.
\end{itemize}
\end{thm}

\begin{proof} 
Let $c\in\C$ such that $B(c)\ne0$. 
First of all, it is without loss of generality to suppose that 
$\Delta(c)\ne0$.
In fact, on the one hand, 
if $c$ is a common zero, then \cref{lem:<>0} implies $\Delta(c)\ne0$.
On the other hand, Condition (i) implies
\[
\Delta(c)\=A^2(c)+4B(c)\=(x+4)B(c)\;\ne\;0.
\] 

By \cref{thm:pr}, the number $c$ is a common zero if and only if 
there exist $p,r\in\Z^+$ such that $v^p=1$ and $v^r=u$.
It is clear that 
\[
v^p=1\eqrl
\begin{cases}
\,|v|^p=1,\\[3pt]
\,p\cdot \Arg(v)=2k\pi,\qquad\text{for some $k\in\Z$}.
\end{cases}
\]
Equally clear is that
\[
v^r=u\eqrl
\begin{cases}
\,|v|^r=|u|,\\[3pt]
\,\Arg(u)=r\cdot \Arg(v)+2k'\pi,\qquad\text{for some $k'\in\Z$}.
\end{cases}
\]
The equation system consisting of these 4 equations can be recast as 
\begin{align}
&|v|=|u|=1,\notag\\
&\Arg(v)=2q\pi/p,\qquad\text{for some $q\in\Z$ such that $(p,\,q)=1$},
\label{Arg:v}\\
&\Arg(u)=2qr\pi/p+2q'\pi,\qquad\text{for some $q'\in\Z$}.
\label{Arg:u}
\end{align}
We notice the following equivalence relation for any $\alpha,\beta\in\C$ such that $\beta\ne0$:
\begin{equation}\label[eqrl]{eqrl:perp}
|\alpha+\beta|=|\alpha-\beta|
\eqrl
\alpha\perp\beta
\eqrl
\alpha^2=y\beta^2\quad\text{for some $y\le 0$.}
\end{equation}
Now we show the desired equivalence in 4 steps.

\begin{step}\label{step:|v|=1}
{\bf (i) $\Leftrightarrow |v|=1$.}
From definition of $v$ and by \cref{eqrl:perp}, we have
\begin{align}
|v|=1&\eqrl
\bigl|A(c)+\sqrt{\Delta(c)}\,\bigr|=\bigl|A(c)-\sqrt{\Delta(c)}\,\bigr|\notag\\
&\eqrl A(c)\perp \sqrt{\Delta(c)}\label[eqrl]{pf:perp}\\
&\eqrl A^2(c)=x'\cdotp\Delta(c)\qquad\text{for some $x'\le 0$}\notag\\
&\eqrl A^2(c)=x\cdotp B(c)\qquad\text{for some $-4<x\le 0$}.\notag
\end{align}
\end{step}

\begin{step}\label{step:|u|=1}
{\bf If $|v|=1$, then (ii) $\Leftrightarrow |u|=1$.}
From definition of $u$, we have
\begin{equation}\label[eqrl]{eqrl1:u=1}
|u|=1
\eqrl
\bigl|2W_1(c)-A(c)+\sqrt{\Delta(c)}\,\bigr|
=\bigl|2W_1(c)-A(c)-\sqrt{\Delta(c)}\,\bigr|.
\end{equation}
When $2W_1(c)-A(c)=0$, (ii) is straightforward, 
and $|u|=1$ follows from \cref{eqrl1:u=1}.
Suppose that $2W_1(c)-A(c)\ne0$.
By \cref{eqrl1:u=1,pf:perp,eqrl:perp} and~(i), we can deduce that 
\begin{align*}
|u|=1
&\eqrl
\bg{2W_1(c)-A(c)}\perp\sqrt{\Delta(c)}\\
&\eqrl
\bg{2W_1-A(c)}\prl A(c)\\
&\eqrl
A(c)\prl W_1(c).
\end{align*}
\end{step}

\begin{step}
{\bf Necessity.}
\Cref{Arg:v,Arg:u} imply that $\Arg(v),\Arg(u)\in\Q\pi$. 
By the minimality of $p$ and $r$, we can deduce that 
\[
p^*(v)\=p
\rmand
p^*(u)\=\frac{p}{(p,\,r)}.
\]
Thus $p^*(u)\,|\,p^*(v)$.
This completes the proof for necessity.
\end{step}

\begin{step}
{\bf Sufficiency.}
Since $\Arg(u),\Arg(v)\in\Q(\pi)$, we can suppose that 
\begin{align}
\Arg(v)&\=\frac{2q\pi}{p^*(v)},
\qquad\text{where $(p^*(v),\,q)=1$},\rmand\notag\\[5pt]
\Arg(u)&\=\frac{2h\pi}{p^*(u)},
\qquad\text{where $(p^*(u),\,h)=1$}.\label{def:Arg:u}
\end{align}
By the minimality of $p$, we deduce that $p^*(v)=p$.
This proves \cref{Arg:v}. 

Let $S=\{0,\,1,\,\ldots,\,p-1\}$.
Since $p^*(u)$ divides $p$, we can suppose that 
\begin{equation}\label{def:d}
p\=d\cdotp p^*(u),\qquad\text{where $d\in\Z^+$}.
\end{equation}
Since $(p,\,q)=1$, the set $\{qk\colon k\in S\}$ forms a complete residue system modulo $p$. Thus we can define $r'$ to be the unique integer in $S$ such that 
\begin{equation}\label{mod:rq=hd}
qr'\;\equiv\; dh\pmod{p}.
\end{equation}
If $r'=0$, then $p\,|\,dh$.
In view of \cref{def:d}, we infer that $p^*(u)\,|\,h$.
But $(p^*(u),h)=1$, we must have $p^*(u)=1$, 
and thus $u\in\R^+$. 
From \cref{step:|u|=1}, we deduce that $u=1$.
Then the definition of $u$ implies $\Delta(c)=0$, a contradiction.
This proves that 
\[
r'\in\{1,2,\ldots,p-1\}.
\]
Now, by \cref{def:Arg:u,def:d,mod:rq=hd},   we can write
\begin{equation}\label{pf26}
\Arg(u)
\=\frac{2h\pi}{p^*(u)}
\=\frac{2dh\pi}{p}
\=\frac{2qr'\pi}{p}+2q'\pi
\=\Arg\bg{v^{r'}}+2q'\pi
\end{equation}
for some $q'\in\Z$.
Now, from \cref{step:|v|=1,step:|u|=1}, we have $|u|=|v|=1$.
It follows that $u=v^{r'}$.
Since $r$ is the unique integer in the set $\{1,2,\ldots,p\}$ such that $u=v^r$, we deduce that $r=r'$.
Thus \cref{pf26} becomes \cref{Arg:u}.
\end{step}
This completes the proof.
\end{proof}

We present two examples.
\begin{eg}\label{eg1}
Suppose that $W_0(z)=1$, $W_1(z)=z$, and 
\[
W_n(z)
=\bg{z^2+(1-\sqrt{3})z+1}W_{n-1}(z)-\frac{z^2}{2}\cdot W_{n-2}(z),
\]
for $n\ge2$. Then we have $W_{n}(e^{\pi i/6})=0$ if and only if
$n\equiv3\pmod{4}$.
\end{eg}
\begin{proof}
Let $c=e^{\pi i/6}$. Then $A(c)=c$ and $\Delta(c)=-c^2$. 
It follows that $v=-i$ and $u=i$. Hence $p^*(v)=4$ and $p^*(u)=4$.
By \cref{thm:cr}, the number $c$ is a common zero.
By \cref{thm:pr}, we have $p=4$ and $r=3$.
The desired equivalence thus follows from \cref{thm:n}.
\end{proof}

In \cref{eg:Ccoeff}, 
the coefficients of the polynomial $B(z)$ are not all real. 
\begin{eg}\label{eg:Ccoeff}
Suppose that $W_0(z)=1$, $W_1(z)=z$, and 
\[
W_n(z)=(4z^2+1)W_{n-1}(z)+e^{-2\pi i/3}\cdot z\cdot W_{n-2}(z), 
\]
for $n\ge2$. Then we have
$W_{n}(e^{\pi i/3}/2)=0$ 
if and only if
$n\equiv2\pmod{4}$.
\end{eg}

The next corollary gives further information for the argument $\Arg(v)$.

\begin{cor}\label{cor:arg}
Let $c$ be a common zero of $\W$ such that $B(c)\ne0$.
Then there is an integer $q$ such that $(p,\,q)=1$ and 
\[
\Arg(A(c)+\sqrt{\Delta(c)})-\Arg\bg{A(c)}\=\frac{q\pi}{p}.
\]
\end{cor}

\begin{proof}
For any $c\in\C$, we define
\begin{equation}\label{def:theta}
\theta\=\Arg(A(c)+\sqrt{\Delta(c)}).
\end{equation}
For any nonzero complex numbers $\alpha$ and $\beta$ 
such that $\alpha\perp\beta$, 
the vectors $\alpha+\beta$ and $\alpha-\beta$ are symmetric 
about the vector $\alpha$. 
Thus the radian obtained by rotating the vector $\alpha$ to $\alpha+\beta$ equals the radian obtained by rotating the vector $\alpha-\beta$ to $\alpha$. In other words, we have 
\begin{equation}\label{arg:sym}
\Arg(\alpha+\beta)-\Arg(\alpha)\=\Arg(\alpha)-\Arg(\alpha-\beta)+2k\pi\qquad\text{for some $k\in\Z$}.
\end{equation}

As a specialization, 
taking $\alpha=A(c)$ and $\beta=\sqrt{\Delta(c)}$.
From \cref{thm:cr}, we see that $\alpha\perp\beta$.
Then \cref{arg:sym} implies that 
\[
\Arg\bg{A(c)-\sqrt{\Delta(c)}}\=2\Arg(A(c))-\theta+2k\pi
\qquad\text{for some $k\in\Z$}.
\] 
Thus, from definition of $v$, we have
\begin{align*}
\Arg(v)
&\=\Arg\bg{A(c)+\sqrt{\Delta(c)}\,}-\Arg\bg{A(c)-\sqrt{\Delta(c)}\,}+2k\pi\\
&\=\theta-\bg{2\Arg(A(c))-\theta}+2k'\pi\\
&\=2\bg{\theta-\Arg(A(c))}+2k'\pi
\end{align*}
for some $k,k'\in\Z$.
Now, extracting the argument from both sides of $v^p=1$, 
we obtain that 
\begin{equation}\label{def:q}
p\cdot 2\bg{\theta-\Arg(A(c))}\=q\cdot 2\pi,
\qquad\text{for some $q\in \Z$}.
\end{equation}
It follows that $\theta-\Arg(A(c))=q\pi/p$. 

Dividing both sides of \cref{def:q} 
by the greatest common divisor $(p,q)$ yields
\begin{equation}\label{pq}
\frac{p}{(p,q)}\Bg{\theta-\bg{2\Arg(A(c))-\theta}}\=\frac{q}{(p,q)}\cdotp 2\pi.
\end{equation}
Since $|v|=1$, \cref{pq} implies that
$v^{p/(p,q)}=1$.
From the minimality of $p$, we infer that $(p,q)=1$.
\end{proof}

\section{Real common zeros of polynomials with real coefficients}\label[sec]{sec:real}

When the polynomials $A(z)$, $B(z)$, and $W_n(z)$ are with real coefficients, results in \cref{sec:complex} reduce,
and we can say something more about real common zeros.

\begin{thm}\label{thm:cr:real}
Suppose that the polynomials $A(z)$, $B(z)$, and $W_1(z)$
are with real coefficients.
Let $c\in\R$ such that $B(c)\ne0$.
Then $c$ is a common zero of $\W$ if and only if 
one of the following cases happen:
\begin{itemize}
\itemsep 5pt
\item
$\Delta(c)>0$, $A(c)=0$, and $W_{n}(c)=0$ if and only if $n$ is odd;
\item
$\Delta(c)<0$, and Condition (iii) in \cref{thm:cr} holds.
\end{itemize}  
\end{thm}
\begin{proof}
Since $A(z)$ and $B(z)$ are polynomials with real coefficients,
and $c\in\R$, we have $\Delta(c)\in\R$.
We deal with the sufficiency and necessity individually.

\noindent{\bf The ``only if'' part.}
Suppose that $c$ is a common zero.
By \cref{lem:<>0}, we have $\Delta(c)\ne0$.
If $\Delta(c)<0$, then \cref{thm:cr} implies Condition (iii).
Otherwise $\Delta(c)>0$, then $v\in\R$ from the definition of $v$.
Consequently, the equation $v^p=1$ implies that $v\in\{\pm1\}$.
It follows that $A(c)=0$ and thus $v=-1$.
Therefore $p=2$ from definition. 
Since $r\in\{1,2,\ldots,p-1\}$, we find $r=1$.
By \cref{thm:n}, we have $W_n(c)=0$ if and only if $n$ is odd.

\noindent{\bf The ``if'' part.}
If $W_n(c)=0$ for odd $n$, then $c$ is a common zero from definition.
We can suppose that $\Delta(c)<0$ and Condition (iii) holds true.
Then it suffices to verify Conditions (i) and (ii) in \cref{thm:cr}.
In fact, since $\Delta(c)<0$, we have $B(c)<0$ from the definition of $\Delta(z)$.
Since $A^2(c)\ge0$, there exists $x<0$ such that $A^2(c)=xB(c)$.
Thus $0>\Delta(c)=(x+4)B(c)$. Therefore, $x>-4$. This proves (i).
On the other hand, since the polynomial $W_1(z)$ is with real coefficients, we have $W_1(c)\in\R$. 
It follows that $A(c)\prl W_1(c)$, i.e., (ii) holds true.
\end{proof}

When $A(c)\in\R$ and $\Delta(c)<0$, the argument $\theta$ defined by \cref{def:theta} can be defined alternatively as $\theta\in(0,\pi)$ such that 
\begin{equation}\label{def:theta:real}
\tan\theta\=\frac{\sqrt{-\Delta(c)}}{A(c)}.
\end{equation}
In particular, we have $\theta=\pi/2$ if $A(c)=0$.
In this case, the general solution to the number sequence $\{W_n\}$ in \cref{lem:00} can be recast in terms of $\theta$; see \cref{lem:00:real}.

\begin{lem}\label{lem:00:real}
Let $\{W_n\}_{n\ge0}$ be a sequence of real numbers defined 
by $W_0=1$ and $W_n=AW_{n-1}+BW_{n-2}$ for $n\ge 2$,
where $A,B\in\R$.
If $\Delta=A^2+4B<0$, 
then
\[
W_n
\=\bigl|B^{n/2}\bigr|\cdot
\bgg{\cos{n\theta}+{(2W_1-A)\sin{n\theta}\over \sqrt{-\Delta}}},
\qquad\text{for $n\ge 0$}.
\]
\end{lem}
\begin{proof}
It can be found in \cite{GMTW16-01,GMTW16-10} that 
\[
W_n
\=\frac{R^n}{2^n}
\bgg{\cos{n\theta}+{(2W_1-A)\sin{n\theta}\over \sqrt{-\Delta}}},
\]
if $A\ne0$ and $|A+\sqrt{\Delta}|=R$.
It is direct to verify that $R=2\sqrt{|B|}$ and that the above expression is  true if $A=0$. This completes the proof.
\end{proof}

\Cref{lem:00:real} gives more information about the argument.

\begin{thm}\label{thm:tan:rtheta}
Suppose that the polynomials $A(z)$, $B(z)$, and $W_1(z)$
are with real coefficients.
Let $c\in\R$ be a common zero of $\W$ such that $A(c)B(c)\ne0$.
If $A(c)-2W_1(c)\ne0$, then
\begin{equation}\label{thm:tan:rtheta}
\tan(r\theta)=\frac{\sqrt{-\Delta(c)}}{A(c)-2W_1(c)},
\end{equation}
where $r$ is defined by \cref{thm:pr}, 
and $\theta\in(0,\pi)$ is defined by \cref{def:theta:real}.
\end{thm}

\begin{proof}
By \cref{thm:cr:real}, we have $\Delta(c)<0$.
By \cref{thm:n}, we have $W_{r}(c)=0$.
From \cref{lem:00:real}, we obtain that 
\[
\cos(r\theta)+{\bg{2W_1(c)-A(c)}\sin(r\theta)\over \sqrt{-\Delta(c)}}=0.
\]
Since $2W_1(c)-A(c)\ne0$, the above equation implies  \cref{thm:tan:rtheta} immediately. 
\end{proof}

\begin{cor}\label{cor:cr:A=ax+b}
Let $\{W_n(z)\}_{n\ge 0}$ be a normalized recursive polynomial sequence of order two defined by \cref{rec:order2} with $W_1(z)=z$,
where $A(z)=az+b$ with $a,b\in\R$,
and $B(z)$ is a polynomial with real coefficients such that $B(0)\ne0$.
Suppose that $c\in\R$ is a common zero such that $A(c)\ne0$.
Then 
\[
c\=\frac{b\bg{\tan\theta-\tan(r\theta)}}
{(a-2)\tan(r\theta)-a\tan\theta},
\]
where $r$ is defined by \cref{thm:pr}, 
and $\theta\in(0,\pi)$ is defined by \cref{def:theta:real}.
\end{cor}

\begin{proof}
From the premise, we have $A(c)=ac+b\ne0$.
Suppose to the contrary that $B(c)=0$.
By \cref{lem:AB=0}, we have $W_2(c)=0$.
By the recurrence relation, it is routine to calculate that
$W_2(z)=A(z)z+B(z)$.
It follows that $0=W_2(c)=A(c)c$.
Since $A(c)\ne0$, we infer that $c=0$, and therefore $B(c)=B(0)\ne0$, a contradiction!
This proves $B(c)\ne0$.
Now we can use \cref{thm:cr:real}.
Note that $W_1(c)=c$.
Dividing \cref{thm:tan:rtheta} by \cref{def:theta:real},
we can obtain the desired expression of $c$ by solving it out.
This completes the proof.
\end{proof}	

\begin{eg}\label{eg:cr:r=2}
Let $\theta\in\Q\pi\cap(0,\,\pi/2)$.
Let $a,b\in\R$ such that $a\ne4\cos^2\theta$ and $b\ne0$.
Suppose that $W_0(z)=1$, $W_1(z)=z$, and 
\begin{equation}\label{eg:r=2}
W_n(z)\=(az+b)W_{n-1}(z)-\frac{4b^2\cos^2\theta}{(4\cos^2\theta-a)^2}W_{n-2}(z),
\end{equation}
for $n\ge 2$. By \cref{cor:cr:A=ax+b}, we have 
\[
W_n\bgg{\frac{b}{4\cos^2\theta-a}}=0
\eqrl n\equiv2\pmod{p^*(\theta)}.
\]
\end{eg}

Particular cases of \cref{eg:cr:r=2} include the followings.

\begin{itemize}
\itemsep 5pt
\item{\bf Case $p^*(\theta)=3$.}
If $a\ne1$, and if \cref{eg:r=2} is 
\[
W_n(z)=(az+b)W_{n-1}(z)-\frac{b^2}{(1-a)^2}W_{n-2}(z),
\]
then $W_n(b/(1-a))=0$ if and only of $n\equiv 2\pmod{3}$.
\item{\bf Case $p^*(\theta)=4$.}
If $a\ne2$, and if \cref{eg:r=2} is 
\[
W_n(z)=(az+b)W_{n-1}(z)-\frac{2b^2}{(2-a)^2}W_{n-2}(z),
\]
then $W_n(b/(2-a))=0$ if and only if $n\equiv 2\pmod{4}$.
\item{\bf Case $p^*(\theta)=6$.}
If $a\ne3$, and if \cref{eg:r=2} is 
\[
W_n(z)=(az+b)W_{n-1}(z)-\frac{3b^2}{(3-a)^2}W_{n-2}(z),
\]
then $W_n(b/(3-a))=0$ if and only if $n\equiv 2\pmod{6}$.
\end{itemize}

\end{document}